\documentclass[10pt, a4paper, english]{smfart}
\usepackage[english]{babel}
\usepackage{smfthm}
\usepackage[numbers]{natbib}
\usepackage{amssymb}
\usepackage{tikz-cd}
\usepackage{mathtools}
\usepackage{enumitem}
\usepackage{lipsum}

\DeclareMathOperator{\GL}{GL}

\DeclareMathOperator{\End}{End}

\DeclareMathOperator{\MT}{MT}
\DeclareMathOperator{\tot}{tot}

\DeclareMathOperator{\SO}{SO}
\renewcommand{\footnotesize}{\scriptsize}

 \title{On the Mumford--Tate conjecture for hyperk\"{a}hler varieties}
 \author{Salvatore Floccari}
 \address{Radboud University, IMAPP, Nijmegen, The Netherlands.\newline }
 \email{s.floccari@math.ru.nl}
  	
\begin{document}
 
  
\maketitle
\begin{abstract}
	We study the Mumford--Tate conjecture for hyperk\"{a}hler varieties. We show that the full conjecture holds for all varieties deformation equivalent to either an Hilbert scheme of points on a K3 surface or to O'Grady's ten dimensional example, and all of their self-products. For an arbitrary hyperk\"{a}hler variety whose second Betti number is not 3, we prove the Mumford--Tate conjecture in every codimension under the assumption that the K\"{u}nneth components in even degree of its Andr\'{e} motive are abelian. Our results extend a theorem of Andr\'{e}.
\end{abstract}

	
 	\section{Introduction}
 	\setcounter{subsection}{-1}
 	\subsection{} 
 	Let $k \subset \mathbb{C}$ be a finitely generated field, with algebraic closure $\bar{k} \subset \mathbb{C}$, and let $\ell$ be a prime number. Given a smooth and projective variety $X$ over $k$, Artin's comparison theorem gives a canonical identification of $\mathbb{Q}_\ell$-vector spaces \[
 	H^i_{\text{B}}\bigl(X(\mathbb{C}),\mathbb{Q}\bigr) \otimes_\mathbb{Q} \mathbb{Q}_\ell \cong H^i_{\textit{\'{e}t}}\bigl(X_{\bar{k}},\mathbb{Q}_\ell\bigr)
 	\]
 	between singular cohomology groups of $X(\mathbb{C})$ and $\ell$-adic cohomology groups of $X_{\bar{k}}$. 
 	
 	Both sides come with additional structure, namely, a Hodge structure on the left hand side and a Galois representation on the right hand side. These data are encoded in the corresponding tannakian fundamental groups. The Mumford--Tate conjecture predicts that Artin's comparison isomorphism identifies the two groups. We refer to this statement for $i=2j$ as the Mumford--Tate conjecture in codimension~$j$ for~$X$. 
 	
 	The Mumford--Tate conjecture is a difficult open problem. It is known only in a very limited number of cases, see \cite[\S2.4, \S3.3, \S4.4]{moonen17} for a recent survey. 	
 	
 	\subsection{Results}

 Our main result establishes the Mumford--Tate conjecture for hyperk\"{a}hler varieties $X$ over $k$ that are of $\mathrm{K}3^{[m]}$ or $\mathrm{OG}10$-type, \textit{i.e.}, such that the complex manifold~$X(\mathbb{C})$ is a deformation of the Hilbert scheme of zero-dimensional subschemes of length~$m$ on a $\mathrm{K}3$ surface \cite{beauville1983varietes} or of O'Grady's ten dimensional hyperk\"{a}hler manifold \cite{O'G99}. The second Betti number of $X$ is $23$ in the first case and $24$ in the second.
 		
 	\begin{theo}\label{mtcK3type}  
 		Let $X$ be a hyperk\"{a}hler variety of either $\mathrm{K3}^{[m]}$ or $\mathrm{OG}10$-type.
 		Then, the Mumford--Tate conjecture holds in any codimension for $X$ and for all self-products $X^{j}$.
 	\end{theo}
 	
 	Our second result establishes the Mumford--Tate conjecture in any codimension for a hyperk\"{a}hler variety with $b_2>3$ whose even Andr\'{e} motive is abelian.
 	
 	\begin{theo}\label{abelianityMTC}
 		Let $X$ be a hyperk\"{a}hler variety such that $b_2(X)>3$. Assume that, for all $i\geq 0$, the even component of the Andr\'{e} motive of $X$ is an abelian motive. Then, the Mumford--Tate conjecture holds in any codimension for $X$. In particular, the Hodge and Tate conjecture for $X$ are equivalent.
 	\end{theo}
 
 The work \cite{KSV2017} suggests that any hyperk\"{a}hler variety has abelian Andr\'{e} motive, however, for the time being, this statement remains a conjecture (but see \textbf{\S\ref{subsec:relatedWork}}). 
 By definition, the second Betti number of a hyperk\"{a}hler variety is always at least $3$; all known examples satisfy $b_2>3$ and it is believed that no hyperk\"{a}hler variety with~$b_2=3$ exists \cite[Question 4]{beauville2011holomorphic}.

 	\subsection{Overview of the contents} 	
 	 We recall in \S\textbf{\ref{section:Mumford-Tate}} the statement of the Mumford--Tate conjecture and its motivic version; throughout, we will use the category of motives constructed by Andr\'{e} in \cite{andre1996Motives}. The following theorem is essentially proven in~\cite{Andre1996}, and it has been generalized in~\cite{moonen2017}.
 	\begin{theo}[Andr\'{e}]\label{codimension1}
 		Let $X$ be a hyperk\"{a}hler variety such that $b_2(X)>3$. Then, the motivic Mumford--Tate conjecture in codimension $1$ holds for~$X$.
 	\end{theo}
    The main tool used in the proof of this result is the Kuga-Satake construction in families, building on ideas due to Deligne \cite{deligne1971conjecture}. The assumption that $b_2(X)>3$ ensures the existence of non-trivial deformations of $X$, as otherwise the moduli space of hyperk\"{a}hler varieties deformation equivalent to $X$ would be zero-dimensional.

    With $X$ as above, we consider the even part $H_{\mathrm{B}}^+(X)$ of the singular cohomology with rational coefficients of $X(\mathbb{C})$, so $H_{\mathrm{B}}^+(X)=\bigoplus_i H_{\mathrm{B}}^{2i}(X)$. A crucial ingredient for us is the action of a~$\mathbb{Q}$-Lie algebra~$\mathfrak{g}_{\tot}(X)$ on $H_{\mathrm{B}}^+(X)$. This construction is due to Verbitsky \cite{verbitsky1996cohomology} and Looijenga-Lunts \cite{looijenga1997lie}; we recall it in \S\textbf{3}. The even singular cohomology of $X$ is the Hodge realization of a motive $\mathcal{H}^+(X)$, whose motivic Galois group is denoted by~$\mathrm{G}_{\mathrm{mot}}^+(X)$. We study the interplay between the actions of this group and of the Lie algebra $\mathfrak{g}_{\tot}(X)$ on $H_{\mathrm{B}}^+(X)$. This cohomology algebra carries a Hodge structure, whose Mumford--Tate group is denoted by $\MT^+(X)$; we show in \S\textbf{\ref{section:asplitting}} that $\MT^+(X)$ is a direct factor of the motivic Galois group~$\mathrm{G}^+_{\mathrm{mot}}(X)$. Here, we need to assume that $b_2(X)>3$ since we use Andr\'{e}'s Theorem \ref{codimension1}.
    
    In~\S\textbf{\ref{section:half}} we prove that if $\MT^+(X)$ has finite index in the motivic Galois group $\mathrm{G}_{\mathrm{mot}}^+(X)$, then the Mumford--Tate conjecture holds in arbitrary codimension for $X$, see Proposition \ref{prop:criterion}. The proof of Theorem \ref{abelianityMTC} is given in \S\textbf{\ref{proofsidethm}}; this is in fact a direct consequence of the proposition and a general result on abelian motives due to Andr\'{e}.
    
    In \S\textbf{\ref{k3m}} we complete the proof of our main result Theorem \ref{mtcK3type}. By Proposition~\ref{thesplit}, $\MT^+(X)$ is a direct product factor of $\mathrm{G}_{\mathrm{mot}}^+(X)$; moreover, the complement satisfies various constraints and in particular it commutes with the action of $\mathfrak{g}_{\tot}(X)$, see Lemma \ref{commuta}. 
    For the $\mathrm{K}3^{[m]}$-type, we have a very effective understanding of this action thanks to work of Markman \cite{markman2008}, and we deduce from his results that the Mumford--Tate group has finite index in the motivic Galois group. For the $\mathrm{OG}10$-type, this finiteness follows instead from the complete description of the $\mathfrak{g}_{\tot}(X)$-representation on the cohomology given by Green-Kim-Laza-Robles \cite{green2019llv}. In both cases, we apply Proposition \ref{prop:criterion} to conclude.
    
    \subsection{Related works}\label{subsec:relatedWork}
    
    The abelianity of the Andr\'e motives of varieties of deformation type $\mathrm{K}3^{[m]}$, $\mathrm{Kum}_m$ and $\mathrm{OG}6$ has been recently established by Soldatenkov in~\cite{soldatenkov19}; our Theorem \ref{abelianityMTC} then implies the Mumford--Tate conjecture in arbitrary codimension for these varieties. 
    Successively, together with Lie Fu and Ziyu Zhang we have shown in \cite{floccarifuzhang} the abelianity of the Andr\'{e} motives of varieties of $\mathrm{OG}10$-type, the fourth and last known deformation type of hyperk\"{a}hler manifolds, and established the full statement of the Mumford--Tate conjecture for all products of hyperk\"{a}hler varieties of known deformation type.
    
    In each case, the proof requires a deformation to an explicit example in the given deformation type. We remark that the proof of Theorem \ref{mtcK3type} presented here is different and simpler: it uses neither deformation to a specific example, nor abelianity of the motives involved. We hope that a refinement of this method might lead to a proof of the Mumford--Tate conjecture for arbitrary hyperk\"{a}hler varieties with $b_2> 3$.
   
 \subsection{Notation and conventions}\label{notations}
 Throughout the whole text, $k\subset\mathbb{C}$ will be a finitely generated field with algebraic closure $\bar{k}\subset \mathbb{C}$, and $\ell$ will be a fixed prime number. A hyperk\"{a}hler variety over $k$ is a smooth projective variety over $k$ such that $X(\mathbb{C})$ is a hyperk\"{a}hler manifold, as defined in \S\textbf{\ref{subsection:2.0}}. Given a complex variety $X$, we denote by~$H^{i}(X)$ its rational singular cohomology groups. The word ``motive" always indicates an object of Andr\'{e}'s category of motives (see~\S\textbf{\ref{motives}}).

\subsection*{Acknowledgements} I am most grateful to Ben Moonen and Arne Smeets for their careful reading and the many comments, which substantially improved this~text. I am also thankful to the anonymous referee for his/her comments.
 	
 	\section{The Mumford--Tate conjecture}\label{section:Mumford-Tate}
 	\setcounter{subsection}{-1}
 	\subsection{}\label{subsection:1.0}
 	We refer to \cite{moonen17} and the references therein. With notations and assumptions as in \S\textbf{\ref{notations}}, we let $X$ be a smooth projective variety over the field $k$. We can extract information about $X$ by looking at various cohomology groups.
 	
 	\subsection{Betti cohomology}\label{subsection:bcohomology}
 	We denote by $H^{i}_{\text{B}}(X)$ the $i$-th singular cohomology group with rational coefficients of the complex manifold $X(\mathbb{C})$. It carries a pure polarizable $\mathbb{Q}$-Hodge structure of weight $i$. Associated to $H^{i}_{\text{B}}(X)$ is its Mumford--Tate group~$\MT\bigl(H^{i}_{\text{B}}(X)\bigr)$; it is a reductive, connected algebraic subgroup of~$\GL\bigl(H_{\text{B}}^{i}(X)\bigr)$.
 	
 	\subsection{$\ell$-adic cohomology} \label{subsection:ellcohomology}
 	We write $H^{i}_{\ell}(X)$ for the $i$-th \'{e}tale cohomology group of $X_{\bar{k}}$ with $\mathbb{Q}_{\ell}$-coefficients, which comes with a continuous representation~$\sigma_{\ell}\colon\mathrm{Gal}(\bar{k}/k)\to\GL\bigl(H^{i}_{\ell}(X)\bigr)$; we denote by $\mathcal{G}\bigl(H_{\ell}^{i}(X)\bigr)$ the Zariski closure of the image of $\sigma_{\ell}$. It is an algebraic group over~$\mathbb{Q}_{\ell}$. If $k'/k$ is a field extension, and if $X_{k'}$ denotes the base change of~$X$ to~$k'$, it may happen that $\mathcal{G} \bigl(H^{i}_{\ell}(X_{{k'}})\bigr)$ becomes smaller than~$\mathcal{G} \bigl(H^{i}_{\ell}(X)\bigr)$; however, the connected component of the identity~$\mathcal{G}\bigl(H^{i}_{{\ell}}(X)\bigr)^0$ is stable under finite field extensions, and there exists a finite field extension~$k'/k$ such that $\mathcal{G}\bigl(H^{i}_{\ell}(X_{k'})\bigr)$ becomes connected.
 	
 	\subsection{}\label{subsection:Artincomparison} 
 	Artin's comparison theorem states that, for all $X$ and $i$ as above, there is a canonical isomorphism of $\mathbb{Q}_{\ell}$-vector spaces
 	\[
 	H_{\text{B}}^{i}(X)\otimes\mathbb{Q}_{\ell}\cong H^{i}_{{\ell}}(X) .
 	\]
 	
 	\begin{conj}[Mumford--Tate]\label{mtc}
 		Under the isomorphism of algebraic groups~$\GL\bigl(H_{\mathrm{B}}^{i}(X) \bigr)\otimes\mathbb{Q}_{\ell} \cong \GL\bigl(H^{i}_{{\ell}}(X) \bigr)$ induced by Artin's isomorphism, we have 
 		$$\MT\bigl(H^{i}_{\mathrm{B}}(X)\bigr)\otimes\mathbb{Q}_{\ell}=\mathcal{G} \bigl(H^{i}_{{\ell}}(X)\bigr)^0 .$$
 		The Mumford--Tate conjecture in codimension $j$ for $X$ is this statement for $i=2j$.
 	\end{conj}
 	
 	\subsection{Motives}\label{motives}
 	A third algebraic group is often useful in order to compare the two groups involved in the Mumford--Tate conjecture. Let $\mathrm{Mot}_k$ be the category of Andr\'{e} motives over $k$ from~\cite{andre1996Motives}; it is a $\mathbb{Q}$-linear neutral tannakian semisimple category. We will denote motives by calligraphic letters. Let $\mathcal{M}\in \mathrm{Mot}_k$. For a field extension~$k'/k$, we let~$\mathcal{M}_{k'}$ be the motive over $k'$ obtained from $\mathcal{M}$ via base change.
 	
 	\subsection{}\label{subsection:gmot}
 	The inclusion $k\subset\mathbb{C}$ determines a realization functor from $\mathrm{Mot}_k$ to the category of polarizable $\mathbb{Q}$-Hodge structures, and we write $\mathcal{M}_{\text{B}}$ for the Hodge realization of the motive~$\mathcal{M}$. The composition with the forgetful functor to $\mathbb{Q}$-vector spaces is a fibre functor on $\mathrm{Mot}_k$; the tannakian formalism then yields a reductive~$\mathbb{Q}$-algebraic group $\mathrm{G}_{\mathrm{mot}}(\mathcal{M})$, which is a subgroup of~$\GL(\mathcal{M}_{\text{B}})$. The tannakian subcategory $\langle \mathcal{M}\rangle^{\otimes}$ of $\mathrm{Mot}_{k}$ generated by $\mathcal{M}$ is equivalent to the category of finite dimensional representations of $\mathrm{G}_{\mathrm{mot}}(\mathcal{M})$. We call this group the motivic Galois group of $\mathcal{M}$.
 	
 	\subsection{} \label{subsection:realizations}
 	The prime $\ell$ determines another realization functor to the category of $\ell$-adic Galois representations; we write~$\mathcal{M}_{\ell}$ for the Galois representation attached to the motive $\mathcal{M}$. We obtain an algebraic group $\mathrm{G}_{\mathrm{mot}, \ell}(\mathcal{M})\subset\GL(\mathcal{M}_{\ell})$ over $\mathbb{Q}_{\ell}$ such that the category of its finite dimensional representations is equivalent to $\langle\mathcal{M} \rangle^{\otimes}\otimes\mathbb{Q}_{\ell}$. Artin's comparison theorem yields an isomorphism~$\mathcal{M}_{\text{B}}\otimes\mathbb{Q}_{\ell}\cong \mathcal{M}_{\ell}$,
 	inducing an identification~$
 	\mathrm{G}_{\mathrm{mot}}(\mathcal{M})\otimes\mathbb{Q}_{\ell}=\mathrm{G}_{\mathrm{mot},\ell}(\mathcal{M})
 	$
 	of subgroups of $\GL(\mathcal{M}_{\text{B}})\otimes\mathbb{Q}_{\ell}\cong \GL(\mathcal{M}_{\ell})$.
 	
 	\subsection{}\label{subsection:basechange}
 	We refer to \cite[\S3.1]{moonen17} for an enlightening discussion of the behaviour of $\mathrm{G}_{\mathrm{mot}}(\mathcal{M})$ under extensions of the base field. It suffices to say that there exists a finite field extension $k^{\diamond}/k$ such that $\mathrm{G}_{\mathrm{mot}}(\mathcal{M}_{k'})\cong \mathrm{G}_{\mathrm{mot}}(\mathcal{M}_{k^{\diamond}})$ for all field extensions~$k'/k^{\diamond}$.
 	
 	\begin{conj}[Motivic Mumford--Tate] \label{mtcmot}
 		For any motive $\mathcal{M}\in \mathrm{Mot}_k$, we have 
 		\[
 		\MT\bigl(\mathcal{M}_{\mathrm{B}} \bigr) = \mathrm{G}_{\mathrm{mot}}\bigl(\mathcal{M}_{\bar{k}}\bigr),\ 
 		\mathrm{and}\ 	
 		\mathcal{G}\bigl(\mathcal{M}_{\ell}\bigr)^0 =
 		\mathrm{G}_{\mathrm{mot}, \ell}\bigl(\mathcal{M}_{\bar{k}}\bigr).
 		\] 
 	\end{conj}
 The conjecture is the conjunction of the motivic Hodge and Tate conjectures: the first says that Hodge classes are motivated, hence 
 $ \MT\bigl(\mathcal{M}_{\mathrm{B}} \bigr) = \mathrm{G}_{\mathrm{mot}}\bigl(\mathcal{M}_{\bar{k}}\bigr)$, and the second says that Tate classes are motivated, and hence $\mathcal{G}\bigl(\mathcal{M}_{\ell}\bigr)^0 =
 \mathrm{G}_{\mathrm{mot}, \ell}\bigl(\mathcal{M}_{\bar{k}}\bigr)$. These statements are weak versions of the usual Hodge and Tate conjectures respectively.
  
 We summarize a few known facts about these groups.
 \begin{itemize}
 	\item There are natural inclusions
 	\[
 	\MT(\mathcal{M}_{\text{B}})\subset \mathrm{G}_{\mathrm{mot}}(\mathcal{M}_{\bar{k}})\ \mathrm{and}\ \mathcal{G} (\mathcal{M}_{\ell})^0\subset \mathrm{G}_{\mathrm{mot}}(\mathcal{M}_{\bar{k}})\otimes\mathbb{Q}_{\ell} .
 	\]
 	\item The algebraic group $\MT(\mathcal{M}_{\text{B}})$ is connected and reductive. On the other hand, $\mathcal{G} (\mathcal{M}_{\ell})^0$ is not known to be reductive, while $\mathrm{G}_{\mathrm{mot}}(\mathcal{M}_{\bar{k}})$ is reductive, but not known to be connected in general.
 \end{itemize}
 	
 	\subsection{} \label{subsection:kunnethmotives}
 	There are contravariant functors $\mathcal{H}^{i}$ from the category of smooth projective varieties over $k$ to $\mathrm{Mot}_k$, such that, for any smooth projective variety~$X$ over $k$, we have
 	\[
 	\mathcal{H}^{i}(X)_{\text{B}} = H^{i}_{\text{B}}(X)\ \ \mathrm{and}\ \ \mathcal{H}^{i}(X)_{\ell}= H^{i}_{\ell}(X) .
 	\]
 	Therefore, Conjecture~\ref{mtcmot} implies Conjecture~\ref{mtc}. We refer to Conjecture~\ref{mtcmot} for the motive $\mathcal{H}^{2j}(X)$ as the motivic Mumford--Tate conjecture for $X$ in codimension~$j$.
 	
 	\subsection{Abelian motives} \label{subsection:abelianmotives}
 	A motive $\mathcal{M}\in \mathrm{Mot}_k$ is abelian if it belongs to the tannakian subcategory generated by the motives of all abelian varieties over $k$. We will need the following theorem due to Andr\'{e} \cite{andre1996Motives}, which improves Deligne's result on absolute Hodge classes on abelian varieties from \cite{deligne1982hodge}.
 	
 	\begin{theo}\label{thm:abelianMH}
 		Let $\mathcal{M}\in \mathrm{Mot}_k$ be an abelian motive. Then we have 
 		\[
 		\MT(\mathcal{M}_{\mathrm{B}})=\mathrm{G}_{\mathrm{mot}}(\mathcal{M}_{\bar{k}}) .
 		\]
 		\end{theo}
 	
 	\section{Hyperk\"{a}hler varieties}\label{section:hyperkahler}
 	\setcounter{subsection}{-1}
 	\subsection{} \label{subsection:2.0}
 	In this section, we work over the complex numbers. A hyperk\"{a}hler manifold $X$ is a connected, simply connected, compact K\"{a}hler manifold admitting a nowhere degenerate holomorphic~$2$-form which spans $H^{0,2}(X)$. At times, we use the expression ``hyperk\"{a}hler variety" instead of writing~``projective hyperk\"{a}hler manifold". The dimension of such a manifold is always even; hyperk\"{a}hler surfaces are $\mathrm{K}3$ surfaces. The second cohomology group of a hyperk\"{a}hler manifold $X$ carries a canonical symmetric bilinear form, the Beauville--Bogomolov form, which is non-degenerate and deformation invariant, and yields a morphism of Hodge structures~$H^{2}(X)(1)\otimes H^{2}(X)(1)\to \mathbb{Q}$. We refer to \cite{beauville1983varietes} and \cite{huybrechts1999compact} for a proper introduction to the subject.
 	\medskip
 	
 	Let $X$ be a complex hyperk\"{a}hler variety of dimension $2n$. The rational cohomology~$H^*(X)$ of $X$ is a graded algebra via cup product. Verbitsky and Looijenga-Lunts studied in \cite{verbitsky1996cohomology} and \cite{looijenga1997lie} a Lie algebra action on $H^*(X)$, which we describe below.
 	
 	\subsection{}\label{subsection:jacobsonmorozov}
 	Let $\theta\in \End\bigl(H^*(X)\bigr)$ be the degree $0$ endomorphism whose action on~$H^j(X)$ is multiplication by $j-2n$, for all $j$. Given $x\in H^2(X)$, we denote by $L_x$ the endomorphism of $H^*(X)$ which maps a cohomology class $\alpha$ to the product~$x\wedge \alpha$. We say that a class~$x\in H^2(X)$ has the Lefschetz property if, for all positive integers~$j$, the map~$L^j_x\colon H^{2n-j}(X)\to H^{2n+j}(X)$ is an isomorphism. 
 	The Lefschetz property for~$x\in H^2(X)$ is equivalent to the existence of~$\Lambda_x\in \End\bigl(H^*(X)\bigr)$ such that~$L_x,\theta$, and~$\Lambda_x$ form an~$\mathfrak{sl}_2$-triple, \textit{i.e.}, we have
 	\[
 	[\theta, L_x]=2L_x,\ \ [\theta,\Lambda_x]=-2\Lambda_x\ \ \mathrm{ and} \ \ [L_x,\Lambda_x]=\theta.
 	\]
 	Once it exists, the endomorphism $\Lambda_x$ is uniquely determined, see for instance Proposition~1.4.6 in~\cite[Expos\'{e} X]{giraud1968dix}. 
 	
 	\subsection{} \label{definition}
 	We define $\mathfrak{g}_{\tot}(X)$ as the smallest Lie subalgebra of~$\mathfrak{gl}\bigl(H^*(X)\bigr)$ containing~$L_x$, for all $x\in H^2(X)$, and $\Lambda_x$, for all~$x\in H^{2}(X)$ with the Lefschetz property. The first Chern class of an ample divisor on~$X$ has the Lefschetz property by the Hard Lefschetz theorem. It is shown in \cite[\S(1.9)]{looijenga1997lie} that~$\mathfrak{g}_{\tot}(X)$ is a semisimple~$\mathbb{Q}$-Lie algebra, which is evenly graded by the adjoint action of~$\theta$, so that~$\mathfrak{g}_{\tot}(X)=\bigoplus_i \mathfrak{g}_{2i}(X)$. The action of $\mathfrak{g}_{\tot}(X)$ on the cohomology of $X$ preserves the even and odd cohomology, and the Lie subalgebra~$\mathfrak{g}_0(X)$ consists of the endomorphisms contained in $\mathfrak{g}_{\tot}(X)$ which preserve the grading of $H^*(X)$. The construction does not depend on the complex structure of $X$; therefore,~$\mathfrak{g}_{\tot}(X)$ is deformation~invariant.
 	
 	\subsection{}\label{thestructure}
 	Let now $H$ denote the space~$H^2(X)$ equipped with the Beauville-Bogomolov form. Let $\tilde{H}$ denote the orthogonal direct sum of $H$ with $U=\langle v,w\rangle$ equipped with the form $-2vw$. We summarize the main properties of the Lie algebra~$\mathfrak{g}_{\tot}(X)$.
 	
 	\begin{theo}\label{gtot}
 		\begin{enumerate}
 			\item[$\mathrm{(a)}$] \label{struct}
 			There is an isomorphism $\mathfrak{g}_{\tot}(X)\cong\mathfrak{so}(\tilde{H})$ of $\mathbb{Q}$-Lie algebras, which maps the element $\theta\in\mathfrak{g}_{\tot}(X)$ to the element of $\mathfrak{so}(\tilde{H})$ which acts as multiplication by $-2$ on $v$, by $2$ on $w$, and by $0$ on $H$.
 			\item[$\mathrm{(b)}$] \label{b}
 			We have
 			\[
 			\mathfrak{g}_{\tot}(X)=\mathfrak{g}_{-2}(X)\oplus\mathfrak{g}_{0}(X)\oplus\mathfrak{g}_{2}(X).
 			\]
 			Moreover, $\mathfrak{g}_0(X) \cong \mathfrak{so}(H ) \oplus\mathbb{Q}\cdot\theta$, and $\theta$ is central in $\mathfrak{g}_{0}(X)$. The abelian subalgebra $\mathfrak{g}_{2}(X)$ is the linear span of the endomorphisms $L_x$, and $\mathfrak{g}_{-2}(X)$ is the span of the $\Lambda_x$, for $x\in H^2(X)$ with the Lefschetz property. 
 			\item[$\mathrm{(c)}$] The Lie subalgebra $\mathfrak{g}_0(X)$ acts via derivations on the graded algebra $H^*(X)$. The induced action of $\mathfrak{so}(H)\subset\mathfrak{g}_0(X)$ on $H^2(X)=H$ is the standard representation.
 		\end{enumerate}
 	\end{theo}
 	
 	The above theorem is proven in \cite{verbitsky1996cohomology}, and in~\cite[Proposition~4.5]{looijenga1997lie}. A proof can also be found in the appendix of \cite{KSV2017}. These proofs are carried out with real coefficients, but immediately imply the result with rational coefficients: since $\mathfrak{g}_{\tot}(X)$ is defined over~$\mathbb{Q}$, the equality  
 	$\mathfrak{g}_{\tot}(X)\otimes\mathbb{R}=\mathfrak{so}(\tilde{H})\otimes {\mathbb{R}}$ of Lie subalgebras of $\mathfrak{gl}(\tilde{H})\otimes \mathbb{R}$ shows that the same equality already holds with rational coefficients.
 	
 	\subsection{} \label{subsection:rho}
 	We know from Theorem~\ref{gtot} that the semisimple part of $\mathfrak{g}_{0}(X)$ is isomorphic to~$\mathfrak{so}(H)$. We denote by 
 	\[
 	\rho\colon\mathfrak{so}(H)\to\prod_j \mathfrak{gl}\bigl(H^{j}(X)\bigr)
 	\] 
 	the restriction of the representation $ \mathfrak{g}_0(X)\to \prod_j \mathfrak{gl}\bigl(H^{j}(X)\bigr)$ to the Lie subalgebra~$\mathfrak{so}(H)$. We also let~$\rho^+\colon\mathfrak{so}(H)\to \prod_i \mathfrak{gl}\bigl(H^{2i}(X)\bigr)$ denote the representation induced by $\rho$ on the even cohomology of $X$.
 	
 	\begin{prop} \label{rhoplus}
 		The representation $\rho^+\colon\mathfrak{so}(H)\to \prod_i \mathfrak{gl}\bigl(H^{2i}(X)\bigr)$ integrates to a faithful representation
 		\[
 		\tilde{\rho}^+\colon  \SO(H)\to \prod_i\GL\bigl(H^{2i}(X)\bigr), 
 		\]
 		such that $\pi_2\circ\tilde{\rho}^+\colon\SO(H)\to\GL\bigl(H^2(X)\bigr)=\GL(H)$ is the standard representation, where $\pi_2$ is the obvious projection $\prod_i\GL\bigl(H^{2i}(X )\bigr)\to \GL\bigl(H^2(X )\bigr)$.
 	\end{prop}
 	We refer to \cite[\S8]{verbitsky1995mirror} for a proof.
 	Note that under the representation $\tilde{\rho}^+$, the group~$\SO(H)$ acts via graded algebra automorphisms on the even cohomology of $X$, by part~(c) of Theorem~\ref{gtot}.
 	
 	\subsection{} \label{subsection:weiloperator}
 	We need to recall one more result. Let $W_{\mathbb{C}}\in\End\bigl(H^*(X,\mathbb{C})\bigr)$ be the endomorphism which acts on each $H^{p,q}(X)$ as multiplication by~$i(p-q)$. It is known that $W_{\mathbb{C}}$ is the~$\mathbb{C}$-linear extension of an endomorphism $W\in\End\bigl(H^*(X,\mathbb{R})\bigr)$, which is called the Weil operator. 
 	 	
 	\begin{theo}\label{weil}
 		The Weil operator $W$ is an element of $\rho\bigl(\mathfrak{so}(H)\bigr)\otimes\mathbb{R}$.
 	\end{theo} 
 	This is proven in~\cite{verbitsky1996cohomology}; see also the appendix to the paper \cite{KSV2017}.
 	
 	\subsection{} \label{subsection:startingpoint}
 	Let~$H^+(X)$ denote the weight~$0$ Hodge structure $\bigoplus_i H^{2i}(X)(i)$, and let~$\MT^+(X)$ denote its Mumford--Tate group. Let~$\pi_2\colon \MT^+(X)\twoheadrightarrow\MT\bigl(H^2(X)(1)\bigr)$ be the projection induced by the inclusion of~$H^{2}(X)(1)$ into $ H^+(X)$.
 	We will deduce the following result from Theorem~\ref{weil}.
 	
 	\begin{coro}\label{MT}
 		The map $\pi_2$ is an isomorphism 
 		\[
 		\MT^+(X)\cong \MT\bigl(H^2(X)(1)\bigr).
 		\]
 		In particular, the weight $0$ Hodge structure $H^+(X)$ belongs to the tensor subcategory of polarizable $\mathbb{Q}$-Hodge structures generated by $H^2(X)(1)$.
 	\end{coro}
 	We first prove a lemma.
 	\begin{lemm}\label{contained}
 		We have
 		\[
 		\MT^+(X)\subset \tilde{\rho}^+\bigl(\SO(H)\bigr).
 		\]
 	\end{lemm}
 	\begin{proof}
 		We identify $\SO(H)$ with its image under the representation $\tilde{\rho}^+$ from Proposition~\ref{rhoplus}. Let $T$ be a tensor construction on $H^+(X)$, by which we mean that $T$ is a finite sum 
 		\[
 		T=\bigoplus_i\ \bigl(H^+(X)\bigr)^{\otimes m_i}\otimes\bigl(H^+(X)\bigr)^{\vee,\otimes n_i}.
 		\]
 		for some integers $m_i$ and $n_i$. Both $\MT^+(X)$ and $\SO(H)$ act on the space $T$, as they are both subgroups of~$\GL\bigl(H^+(X)\bigr)$. In order to show that $\MT^+(X)$ is contained into~$\SO(H)$, it suffices to check that, for all tensor constructions $T$ as above, every element~$\alpha$ of~$T$ fixed by the latter is also fixed by $\MT^+(X)$. Indeed, both groups are reductive, and we can then apply~\cite[Proposition~3.1]{deligne1982hodge} to conclude. Let $\alpha\in T$ be invariant for the $\SO(H)$-action. Then, the image of $\alpha$ in $T\otimes\mathbb{C}$ is in the kernel of every element of $\mathfrak{so}(H)\otimes\mathbb{C}$. By Theorem~\ref{weil}, this implies that~$\alpha$ is of type $(0,0)$; hence $\alpha$ is a Hodge class and it is therefore fixed by the Mumford--Tate group.
 	\end{proof}
 	
 	\begin{proof}[Proof of Corollary~\ref{MT}]
 		It suffices to show that the restriction of the projection~$\pi_2$ to $\MT^+(X)$ is injective. The composition~$\pi_2\circ\tilde{\rho}^+\colon\SO(H) \to\GL\bigl(H^2(X)\bigr)$ is injective thanks to Proposition~\ref{rhoplus}. As~$\MT^+(X)\subset\tilde{\rho}^+\bigl(\SO(H)\bigr)$ by Lemma~\ref{contained}, it follows that the restriction of $\pi_2$ to~$\MT^+(X)$ is injective, too.
 	\end{proof}
 	\begin{rema}\label{remark:ok}
 		The conclusion of Corollary~\ref{MT} is true even without the projectivity assumption on $X$, with the only difference that the Hodge structures involved are not necessarily polarizable.
 	\end{rema}
 	
 	\section{A splitting of the motivic Galois group}\label{section:asplitting}
 	\setcounter{subsection}{-1}
 	\subsection{}\label{subsection:3.0}
 	In this section, $X$ is a complex hyperk\"{a}hler variety; we further assume that $b_2(X)>3$. We consider the weight $0$ motive 
 	\[
 	\mathcal{H}^+(X)\coloneqq \bigoplus_i \mathcal{H}^{2i}(X)(i)\in\mathrm{Mot}_{\mathbb{C}},
 	\]
 	and we denote by~$\mathrm{G}^+_{\mathrm{mot}}(X)\subset \prod_i \GL\bigl(H^{2i}(X)(i)\bigr)$ its motivic Galois group. We let~$\bar{\pi}_2$ be the projection~$\mathrm{G}^+_{\mathrm{mot}}(X)\twoheadrightarrow\mathrm{G}_{\mathrm{mot}}\bigl(\mathcal{H}^2(X)(1)\bigr)$ induced by the inclusion of~$\mathcal{H}^{2}(X)(1)$ into $\mathcal{H}^{+}(X)$, and we define
 	\[
 	P(X)\coloneqq \ker (\bar{\pi}_2)\subset\mathrm{G}_{\mathrm{mot}}^+(X).
 	\]
 	\begin{prop}\label{thesplit}
 		We have
 		\[
 		\mathrm{G}_{\mathrm{mot}}^+(X)= P(X)\times\MT^+(X).
 		\]
 	\end{prop}
 	We will first establish some preliminary results.
 	
 	\begin{lemm}\label{section}
 		There exists a section $s$ of the map $\bar{\pi}_2$,
 		\[
 		s\colon \mathrm{G}_{\mathrm{mot}}\bigl(\mathcal{H}^2(X)(1)\bigr)\hookrightarrow\mathrm{G}_{\mathrm{mot}}^+(X),
 		\]
 		whose image coincides with $\MT^+(X)\subset\mathrm{G}_{\mathrm{mot}}^+(X)$.
 	\end{lemm}
 	\begin{proof}
 		We have a commutative diagram
 		\[
 		\begin{tikzcd}
 		\mathrm{G}^+_{\mathrm{mot}}(X)\arrow[two heads]{rr}{\bar{\pi}_2}\arrow[hookleftarrow]{d}{\iota_+} && \mathrm{G}_{\mathrm{mot}}\bigl(\mathcal{H}^2(X)(1)\bigr) \arrow[hookleftarrow]{d}{\iota_2}\\
 		\MT^+(X)\arrow{rr}{{\pi}_2 } && \MT\bigl(H^2(X)(1)\bigr)
 		\end{tikzcd}
 		\]
 		Here, $\iota_+$ and $\iota_2$ denote the natural inclusions; ${\pi}_2$ and $\iota_2$ are isomorphisms due to Corollary~\ref{MT} and Theorem~\ref{codimension1} respectively. We can now take $s=\iota_+\circ( \iota_2 \circ \bar{\pi}_2)^{-1}$.
 	\end{proof}
 	\begin{lemm}\label{commuta}
 	The adjoint action of the group $P(X)\subset\GL\bigl(H^+(X)\bigr)$ on $\mathfrak{gl}\bigl(H^+(X)\bigr)$ restricts to the identity on the Lie algebra $\mathfrak{g}_{\tot}(X)$.
 	\end{lemm}
 \begin{proof}
 	Note that $P(X)$ acts on $H^+(X)$ via algebra automorphisms since the cup-product is induced by an algebraic cycle, namely, the small diagonal $\delta\subset X^3$; moreover, by definition, its action preserves the grading and is trivial on $H^2(X)$. Hence, if $p\in P(X)$, then $p$ commutes with $\theta$ and $L_x$, for $x\in H^2(X)$. Further, if~$x$ has the Lefschetz property, then $p$ commutes with~$\Lambda_x$ as well: indeed,~$L_x$,~$\theta$ and $p\Lambda_xp^{-1}$ form an $\mathfrak{sl}_2$-triple, and this forces $p\Lambda_xp^{-1}=\Lambda_x$, see \S\textbf{\ref{subsection:jacobsonmorozov}}. As the various operators $L_x$ and $\Lambda_x$, for $x\in H^2(X)$, generate the Lie subalgebra $\mathfrak{g}_{\tot}(X)\subset\mathfrak{gl}\bigl(H^+(X)\bigr)$, we conclude that~$P(X)$ commutes with the whole of~$\mathfrak{g}_{\tot}(X)$.
 \end{proof} 
 	\begin{proof}[Proof of Proposition~\ref{thesplit}]
 		By Lemma~\ref{section}, $P(X)\cdot \MT^+(X)=\mathrm{G}_{\mathrm{mot}}^+(X)$, and the two subgroups have trivial intersection. By Lemma~\ref{contained} and the above Lemma~\ref{commuta}, $P(X)$ and $\MT(X)^+$ commute. It follows that~$\mathrm{G}_{\mathrm{mot}}^+(X)$ is the direct product of these two subgroups.
 	\end{proof}
 	
 	\section{A sufficient condition}\label{section:half}
 	 	\setcounter{subsection}{-1}
 	\subsection{}\label{subsection:4.0}
 	With notations and assumptions as in \S\textbf{\ref{notations}}, let $X$ be a hyperk\"{a}hler variety over $k$, and assume that $b_2(X)>3$. Consider the weight~$0$ motive 
 	\[
 	\mathcal{H}^+(X)=\bigoplus_i \mathcal{H}^{2i}(X)(i) \in\mathrm{Mot}_{k},
 	\]
 	and write $\mathrm{G}_{\mathrm{mot}}^+(X)$ for its motivic Galois group. Let $H_{\text{B}}^+(X)$ and~$H_{\ell}^+(X)$ denote respectively the Hodge and $\ell$-adic realization of $\mathcal{H}^+(X)$. We write $\MT^+(X)$ for~$\MT\bigl(H_{\text{B}}^+(X)\bigr)$ and $\mathcal{G}_{\ell}^+(X)$ for $\mathcal{G}\bigl(H^+_{\ell}(X)\bigr)$. We identify $H_{\text{B}}^+(X)\otimes\mathbb{Q}_{\ell}$ with $H^+_{\ell}(X)$ via Artin's comparison isomorphism. Then both~$ \MT^+(X)\otimes\mathbb{Q}_{\ell}$ and $\mathcal{G}_{\ell}^+(X)$ are identified with subgroups of $\GL\bigl(H_{\ell}^+(X)\bigr)$.

 	\subsection{}\label{subsection:main1}
 	Conjecturally, under Artin's isomorphism we have~$\MT^+(X)\otimes \mathbb{Q}_{\ell}\cong \mathcal{G}^+_{\ell}(X)^0$. We refer to this statement as the Mumford--Tate conjecture for $\mathcal{H}^+(X)$; it implies the Mumford--Tate conjecture in codimension $j$ for $X$ and for all integers~$j$, and, if~$X$ has trivial odd cohomology, it also implies the Mumford--Tate conjecture in any codimension for any self-power $X^{k}$. 
 	Recall from \S\textbf{\ref{subsection:basechange}} that~$\mathcal{G}^+_{\ell}(X)^0$ is a subgroup of $
 	\mathrm{G}^+_{\mathrm{mot}}(X_{\bar{k}})\otimes\mathbb{Q}_{\ell}\cong\mathrm{G}^+_{\mathrm{mot}}(X_{\mathbb{C}})\otimes\mathbb{Q}_{\ell}$, and that, by Proposition~\ref{thesplit}, we have an equality~$\mathrm{G}^+_{\mathrm{mot}}(X_{\mathbb{C}})=P(X)\times\MT^+(X)$ of subgroups of $\GL\bigl(H_{\mathrm{B}}^+(X)\bigr)$. 
 	
 	\begin{prop}\label{prop:criterion}
 		Assume that $P(X)$ is finite (resp. trivial). Then the Mumford--Tate conjecture (resp. the motivic Mumford--Tate conjecture) holds for $\mathcal{H}^+(X)$.
 		\end{prop} 
 	\begin{proof}
 		Consider the commutative diagram
 		\[
 		\begin{tikzcd}
 		\MT^+(X)\otimes \mathbb{Q}_{\ell} \arrow[hookrightarrow]{r} \arrow["\sim"{sloped, above}]{d} & \mathrm{G}_{\mathrm{mot}}^+(X_{\bar{k}})\otimes\mathbb{Q}_\ell \arrow[two heads]{d} \arrow[hookleftarrow]{r} & \mathcal{G}^+_{\ell}(X)^0\arrow[two heads]{d}\\
 		\MT\bigl(H_{\text{B}}^2(X)(1)\bigr)\otimes\mathbb{Q}_{\ell}\arrow{r}{\sim} & \mathrm{G}_{\mathrm{mot}}\bigl(\mathcal{H}^2(X_{\bar{k}})(1) \bigr)\otimes\mathbb{Q}_\ell & \mathcal{G} \bigl(H^2_{\ell}(X)(1)\bigr)^0 \arrow[swap]{l}{\sim}
 		\end{tikzcd}
 		\]
 		The horizontal arrows on the bottom are isomorphisms due to Theorem~\ref{codimension1}, and the vertical map on the left is an isomorphism thanks to Corollary~\ref{MT}. By Proposition~\ref{thesplit} we have $\mathrm{G}_{\mathrm{mot}}^+(X_{\bar{k}})=P(X)\times\MT^+(X)$; if $P(X)$ is finite, it follows that we have $\mathrm{G}_{\mathrm{mot}}^+(X_{\bar{k}})^0=\MT^+(X)$. Hence, replacing in the above diagram $\mathrm{G}_{\mathrm{mot}}^+(X_{\bar{k}})$ with its connected component of the identity, also the leftmost arrow on the top row becomes an isomorphism. Thus all arrows in the diagram become isomorphisms, and we obtain 
 		\[
 		\MT^+(X)\otimes\mathbb{Q}_{\ell} = \mathrm{G}_{\mathrm{mot}}^+(X_{\bar{k}})^0\otimes \mathbb{Q}_{\ell} = \mathcal{G}_{\ell}^+(X)^0.
 		\]
 		Moreover, if $P(X)$ is trivial then $\mathrm{G}_{\mathrm{mot}}^+(X_{\bar{k}})$ is connected and equal to $\MT^+(X)$, and therefore the motivic Mumford--Tate conjecture holds in this case.
 	\end{proof}
 
 \subsection{Proof of Theorem \ref{abelianityMTC}}\label{proofsidethm}
 By the above Proposition \ref{prop:criterion}, it suffices to show that the assumption of abelianity of all even K\"{u}nneth components $\mathcal{H}^{2i}(X_{\bar{k}})$ of the motive of $X$ implies that $P(X)$ is trivial. Note that this assumption is equivalent to the abelianity of $\mathcal{H}^+(X_{\bar{k}})$. But then the desired conclusion follows immediately from Proposition \ref{thesplit} and Andr\'{e}'s theorem \ref{thm:abelianMH}: indeed, the first result implies that $\mathrm{G}_{\mathrm{mot}}^+(X_{\bar{k}})= P(X) \times \MT^+(X) $ and the second that $\mathrm{G}_{\mathrm{mot }}^+(X_{\bar{k}})=\MT^+(X)$.
 
 	\section{Proof of Theorem~\ref{mtcK3type}} \label{k3m}
 	 	\setcounter{subsection}{-1}
 	\subsection{}\label{subsection:5.0}
 	
 	In this section we prove Theorem \ref{mtcK3type}. To this end, we will establish the finiteness of the group $P(X)$ from \S\textbf{4.0} when $X$ is a hyperk\"{a}hler variety of~$\mathrm{K}3^{[m]}$ or $\mathrm{OG}10$-type; Theorem \ref{mtcK3type} then follows via Proposition \ref{prop:criterion}. We can, and will, work over the complex numbers, since $\mathrm{G}_{\mathrm{mot}}^+(X_{\bar{k}})\cong \mathrm{G}_{\mathrm{mot}}^+(X_{\mathbb{C}})$.
 	
 	Recall that a complex hyperk\"{a}hler variety $X$ is of $\mathrm{K3}^{[m]}$-type if it is a deformation of a Hilbert scheme of $0$-dimensional subschemes of length $m$ on some $\mathrm{K}3$ surface. If $m=1$, then $X$ is the original $\mathrm{K}3$ surface; we will assume $m\geq 2$. In this case $\dim X=2m$, the odd cohomology of $X$ vanishes, and the second Betti number equals~$23$, \cite{gottsche1990}. 
 	We say that $X$ is of $\mathrm{OG}10$-type if it is deformation equivalent to O'Grady's ten dimensional hyperk\"{a}hler variety constructed in \cite{O'G99}. In this case the odd Betti numbers of $X$ vanish as well, and we have $b_2(X)=24$, see \cite{de2019hodge}. 
 	
 	\subsection{}
 	 Let $\mathrm{Aut}\bigl(H^+(X)\bigr)$ be the group of automorphisms of the graded $\mathbb{Q}$-algebra $H^+(X)=\bigoplus_i H^{2i}(X)$. Let $K(X) \subset \mathrm{Aut}\bigl(H^+(X)\bigr)$ be the kernel of the natural restriction map 
 	 $\mathrm{Aut}\bigl(H^+(X)\bigr) \to \GL\bigl(H^2(X)\bigr)$. The group $P(X)$ acts via algebra automorphisms, and, by construction, its action is trivial in degree $2$. Hence, we have 
 		\[
 		P(X)\subset K(X).
 		\] 
 		To conclude the proof of Theorem \ref{mtcK3type} it therefore suffices to establish the following.
 		
 		\begin{prop}\label{prop:autofiniteness}
 			Assume $X$ is a hyperk\"{a}hler variety of either $\mathrm{K}3^{[m]}$ or $\mathrm{OG}10$-type. Then $K(X)$ is a finite group.
 			\end{prop}
 		
 		As we are going to explain, this is a consequence of results due to Markman \cite{markman2008} in the first case and due to Green-Kim-Laza-Robles \cite{green2019llv} in the second case.
 		
 		\subsection{}
 		We start by recalling from \cite{looijenga1997lie} some additional facts on the representation of $\mathfrak{g}_{\tot}(X)$ on the cohomology. Let $\int_X$ denote the projection $H^+(X)\to H^{4n}(X)\cong\mathbb{Q}$, where $\dim(X)=2n$. Consider the Poincar\'{e} pairing $\phi\colon H^+(X)\otimes H^+(X)\to \mathbb{Q}$, which is defined via
 		\[
 		\phi(\alpha,\beta)= (-1)^{q}\int_{X} \alpha\wedge\beta,
 		\]
 		for $\alpha$ of degree $2n+2q$. It is shown in \cite[Proposition~1.6 and its proof]{looijenga1997lie}, that the Lie algebra $\mathfrak{g}_{\tot}(X)$ preserves infinitesimally the Poincar\'{e} pairing, and that~$\phi$ restricts to a non-degenerate pairing on every $\mathfrak{g}_{\tot}(X)$-submodule of $H^+(X)$.

 		\subsection{} 
 		The group $K(X)$ acts on $H^+(X)$ via graded algebra automorphisms and it acts trivially in degree $2$; it follows that $K(X)$ preserves the pairing $\phi$. Moreover, the argument used to prove Lemma \ref{commuta} shows that this group commutes with $\mathfrak{g}_{\tot}(X)$. 
 		
 		\begin{proof}[Proof of Proposition \ref{prop:autofiniteness} for the $\mathrm{OG}10$-type] 
 		Assume $X$ is of $\mathrm{OG}10$-type. The representation of $\mathfrak{g}_{\tot}(X)$ on the cohomology has been fully described in \cite[Theorem 1.1-(iv)]{green2019llv}. We have 
 		\[
 		H^+(X)= V_1\oplus V_2,
 		\]
 		where $V_1$ is the subalgebra generated by $H^2(X)$ and $V_2$ is an absolutely irreducible $\mathfrak{g}_{\tot}(X)$-representation. We deduce that $K(X)$ is a subgroup of $$\bigl(\End(V_2)^{\mathfrak{g}_{\tot}(X)}\bigr)^{\times}\cong \mathbb{Q}^{\times},$$ by Schur's lemma. Further, we know that the pairing $\phi$ restricts to a non-degenerate invariant form on $V_2$, and we deduce that $K(X)\subset \{1, -1\}$.
 		\end{proof}
 		
 		\subsection{}
 		
 		As apparent from the proof, for the $\mathrm{OG}10$-type the decomposition of the cohomology into $\mathfrak{g}_{\tot}(X)$-isotypical components already imposes the desired finiteness result. The analogous decomposition for the $\mathrm{K}3^{[m]}$-type becomes more and more complicated as the dimension increases, see \cite{green2019llv}. Nevertheless, it becomes more manageable once the algebra structure is taken into account, thanks to the following result of Markman. From now on, we assume that $X$ is a variety of $\mathrm{K}3^{[m]}$-type.
 		
 		For $l\geq 0$, we let $A_{2l} \subset H^+(X)$ be the subalgebra generated by~$\bigoplus_{j\leq l} H^{2j}(X)$. Note that $A_{2l}=H^+(X)$ for $l\geq m$. Recall from Corollary~\ref{rhoplus} that we have a representation~$\tilde{\rho}^+$ of $\SO(H)$ on $H^+(X)$.

 		\begin{theo}[Markman]\label{markmantheorem}
 		 	For all $i\geq 1$, there exists a subspace $C^{2i}\subset H^{2i}(X)$ such with the following properties.
 		 	\begin{enumerate}
 		 		\item[$\mathrm{(a)}$] We have a $\mathfrak{g}_{0}(X)$-invariant decomposition
 		 		\[
 		 		H^{2i}(X)=\bigl(A_{2i-2}\cap H^{2i}(X)\bigr)\oplus C^{2i}.
 		 		\]
 		 		Note that this implies $C^2=H^2(X)$ and $C^{2i}=0$ for $i>m$. Each $C^{2i}$ is in particular a subrepresentation for $\mathfrak{g}_0(X)$ and, hence, for $\SO(H)$. Moreover, the $\mathfrak{g}_{\tot}(X)$-module generated by $C^{2i}$ is orthogonal to $A_{2i-2}$ with respect to $\phi$.
 		 		\item[$\mathrm{(b)}$] The sum $\bigoplus_{i\geq 1} C^{2i}$ generates the algebra $H^+(X)$.
 		 		\item[$\mathrm{(c)}$] The $\SO(H)$-module $C^{2i}$ is a subrepresentation of the sum of a copy of the standard representation with a one dimensional trivial representation, for all $i\geq 2$.
 		 	\end{enumerate}
 		 \end{theo}
 		Parts (a) and (b) are proven in \cite[Corollary~4.6]{markman2008}, while (c) is [\textit{loc. cit.}, Lemma~4.8]. We can now conclude the proof of our main result.
 		
 		\begin{proof}[Proof of Proposition \ref{prop:autofiniteness} for the $\mathrm{K}3^{[m]}$-type]
 		We claim first of all that each subspace $C^{2i}$ is stable under the action of $K(X)$. In fact, since $K(X)$ acts via graded algebra automorphisms, the subalgebras $A_{2l}$ are $K(X)$-stable for all $l$. Since $\phi$ is $K(X)$-invariant, it follows that the orthogonal complement to each $A_{2l}$ is preserved as well; as $K(X)$ acts compatibly with the grading, it indeed stabilizes $C^{2i}$, for all~$i$.
 		
 		The subspaces $C^{2i}$ generate the cohomology by Theorem \ref{markmantheorem}.(b), and $K(X)$ commutes with the representation $\tilde{\rho}$. Hence, we have 
 		\[
 		K(X)\subset \prod_{i\geq  2} \bigl(\End(C^{2i})^{\SO(H)}\bigr)^{\times}.
 		\]
 		Let $V\subset C^{2i}$ be an irreducible-$\SO(H)$ representation. By Theorem \ref{markmantheorem}.(c), the representation $V$ is absolutely irreducible and it appears in $C^{2i}$ with multiplicity one; it follows that $V$ is stable under $K(X)$ as well. By Schur's lemma, each element of $K(X)$ acts on $\mathfrak{g}_{\tot}(X)\cdot V$ via multiplication by some rational number. On the other hand $K(X)$ preserves the form $\phi$, whose restriction to the $\mathfrak{g}_{\tot}(X)$-module generated by $V$ is non-degenerate, and therefore the action of $K(X)$ on $\mathfrak{g}_{\tot}(X)\cdot V$ factors through $\{1, -1\}\cong \mathbb{Z}/2\mathbb{Z}$. Using again Theorem~\ref{markmantheorem}.(c), we conclude that, for all $i$, $\bigl(\End(C^{2i})^{\SO(H)}\bigr)^{\times}$ is a subgroup $\mathbb{Z}/2\mathbb{Z}^2$, and hence we have
 		\[
 		K(X) \subset \prod_{i = 2}^{m} \mathbb{Z}/2\mathbb{Z}^2.
 		\] 		
 		\end{proof}
 		Theorem \ref{mtcK3type} is proved.
 		
 		\begin{rema}
 		The conclusion of Proposition \ref{prop:autofiniteness} does not hold for the remaining deformation types $\mathrm{Kum}_m$ and $\mathrm{OG6}$. This can be checked using the description of the $\mathfrak{g}_{\tot}(X)$-representation of the cohomology given in \cite{green2019llv}: in fact, for these deformation types, there are $\mathfrak{g}_{\tot}(X)$-representations which appear in the cohomology with higher multiplicities, which cannot be explained only by taking into account the algebra structure on the cohomology. 
 		\end{rema}
 	
 	\bibliographystyle{smfalpha}
 	
 	\newcommand{\etalchar}[1]{$^{#1}$}
 	\providecommand{\bysame}{\leavevmode ---\ }
 	\providecommand{\og}{``}
 	\providecommand{\fg}{''}
 	\providecommand{\smfandname}{\&}
 	\providecommand{\smfedsname}{\'eds.}
 	\providecommand{\smfedname}{\'ed.}
 	\providecommand{\smfmastersthesisname}{M\'emoire}
 	\providecommand{\smfphdthesisname}{Th\`ese}

 	\end{document}